\documentclass[reqno,11pt]{amsart}
\numberwithin{equation}{section}

\usepackage[backend=bibtex,style=alphabetic,natbib=true ,url=false]{biblatex}
\usepackage[utf8]{inputenc}
\usepackage[italian,english]{babel}
\usepackage{amsmath}
\usepackage{amsfonts}
\usepackage{amssymb,esint}
\usepackage{tikz}
\usepackage{geometry}
\usepackage{color}
\usepackage{mathrsfs}
\usepackage{mathtools}

\mathtoolsset{showonlyrefs}

\usepackage[colorlinks, citecolor=citegreen, linkcolor=refred, urlcolor=blue]{hyperref}
\definecolor{citegreen}{rgb}{0,0.4,0}
\definecolor{refred}{rgb}{0.5,0,0}

{\catcode `\@=11 \global\let\AddToReset=\@addtoreset}
\AddToReset{equation}{section}

\newcounter{mnotecount}[section]

\renewcommand{\themnotecount}{\thesection.\arabic{mnotecount}}

\newcommand{\mnote}[1]
{\protect{\stepcounter{mnotecount}}$^{\mbox{\footnotesize
$
\bullet$\themnotecount}}$ \marginpar{
\raggedright\tiny\em
$\!\!\!\!\!\!\,\bullet$\themnotecount: #1} }

\newcommand{\jj}[1]%
{{\color{red}\mnote{{\color{red}{\bf jj:} #1} }}}

\theoremstyle{plain}
\newtheorem {theorem}{Theorem}[section]
\newtheorem {lemma}[theorem]{Lemma}
\newtheorem {proposition} [theorem]{Proposition}
\newtheorem {corollary} [theorem]{Corollary}

\theoremstyle{remark}

\DeclarePairedDelimiter\abs{\lvert}{\rvert}

\newcommand{\R}{\mathbb R}
\newcommand{\N}{\mathbb N}

\renewcommand{\theta}{\vartheta}

\newcommand{\barint}
{\rule[.036in]{.12in}{.009in}\kern-.16in \displaystyle\int}

\newcommand{\dive}{{\mathrm{div}}}

\usepackage{color}

\newcommand{\numberset}{\mathbb}
\renewcommand{\N}{\numberset{N}}
\renewcommand{\R}{\numberset{R}}
\newcommand{\Sf}{\numberset{S}}

\newcommand{\ric}{\mathop {\rm Ric}\nolimits}

\newcommand{\dd}{{\,\rm d}}

\renewcommand{\phi}{\varphi}
\renewcommand{\epsilon}{\varepsilon}

\bibliography{example}

\title{A note on the critical Laplace equation and Ricci curvature}

\author[M.~Fogagnolo]{Mattia Fogagnolo}
\address{M.~Fogagnolo, Centro di Ricerca Matematica Ennio De Giorgi, Scuola Normale Superiore,
Piazza dei Cavalieri 3, 56126 Pisa (PI), Italy}
\email{mattia.fogagnolo@sns.it}

\author[A.~Malchiodi]{Andrea Malchiodi}
\address{A.~Malchiodi, Scuola Normale Superiore,
Piazza dei Cavalieri 7, 56126 Pisa (PI), Italy}
\email{andrea.malchiodi@sns.it}

\author[L.~Mazzieri]{Lorenzo Mazzieri}
\address{L.~Mazzieri, Universit\`a degli Studi di Trento,
via Sommarive 14, 38123 Povo (TN), Italy}
\email{lorenzo.mazzieri@unitn.it}

\begin{document}

\begin{abstract}
We study strictly positive solutions to the critical Laplace equation
\[
- \Delta u = n(n-2) u^{\frac{n+2}{n-2}},
\]
decaying at most like $d(o, x)^{-(n-2)/2}$,
on complete noncompact manifolds $(M, g)$ with nonnegative Ricci curvature, of dimension $n \geq 3$. We prove that, under an additional mild assumption on the volume growth, such a solution does not exist, unless $(M, g)$ is isometric to $\R^n$ and $u$ is a Talenti function. The method employs an elementary analysis of a suitable function defined along the level sets of $u$.
\end{abstract}


\maketitle

\bigskip
\noindent\textsc{MSC (2020): 35R01, 35B33, 53C21, 40E10 
. 
}

\smallskip
\noindent{\underline{Keywords}: critical equations, classification results, manifolds with nonnegative Ricci curvature, level sets}

\section{Introduction and statement of the main results}
The critical Laplace equation
\begin{equation}
\label{laplace-intro}
- \Delta u = n(n-2) u^{\frac{n+2}{n-2}},
\end{equation}
on $\R^n$, with $n \geq 3$,
together its possible generalizations,
has been widely studied in the mathematical literature for many reasons. Let us just mention that entire solutions to \eqref{laplace-intro} provide critical functions for the $L^2$ Sobolev inequality and that $u^{4/(n-2)} g_{\R^n}$ is a metric of positive, constant scalar curvature. Since the works of Obata \cite{obata-critlap}, Talenti \cite{talenti}, Aubin \cite{aubin-isoperimetric}, Gidas-Ni-Nirenberg \cite{gidas-ni-nirenberg}, it became clear that under additional decay or variational assumptions the only entire solutions to \eqref{laplace-intro} are radially symmetric. They are often called \emph{Talentian}. The grounbreaking contribution by Caffarelli-Gidas-Spruck \cite{caffarelli-gidas-spruck} actually established that no additional assumptions are required, except for positivity.

\smallskip

In this note, we face the problem of classifying entire solutions to \eqref{laplace-intro}
on complete noncompact manifolds $(M, g)$ with nonnegative Ricci curvature of dimension $n \geq 3$. We prove that, under a suitable decay assumption on $u$ and an additional mild assumption on the volume growth, such a solution does not exist, unless $(M, g)$ is isometric to $\R^n$ and $u$ is a Talenti function.
\begin{theorem}
\label{main}
Let $(M, g)$ be a complete noncompact Riemannian manifold with nonnegative Ricci curvature, of dimension $n \geq 3$, satisfying
\begin{equation}
\label{volume-condition}
\frac{\abs{B(o, t}}{\abs{B(o, s)}} \geq C \left(\frac{t}{s}\right)^b
\end{equation}
for some $1 < b \leq n$ and for any $t \geq s > 0$. Assume that there exists a global, strictly positive solution $u$ to 
\[
- \Delta u = n(n-2) u^{\frac{n+2}{n-2}},
\]
such that $u(x) d(o, x)^{(n-2)/2}$ is uniformly bounded in $M$ as a function of $x \in M$, for some $o \in M$. Then, $(M, g)$ is isometric to $(\R^n, g_{\R^n})$ and 
\begin{equation}
\label{talentian}
u(x) = \left(\frac{a}{a^2 + d(o, x)^2}\right)^{\frac{(n-2)}{2}},
\end{equation}
for some $a > 0$.
\end{theorem}
We observe that the volume condition \eqref{volume-condition} is satisfied any time 
\[
C^{-1} r^b \leq \abs{B(o, r)} \leq C r^b
\]
for $1 < b \leq n$ and for any $r \geq 1$. In particular, Euclidean volume growth manifolds fulfill it.

\smallskip

As a consequence of the previous result, apart from the Euclidean case, the Ricci-flat manifolds $(M, g)$ fulfilling the above assumptions do not enjoy conformal metrics $\tilde{g} = u^{4/(n-2)} g$ of positive constant scalar curvature with conformal factor decaying at infinity as above. In particular, this applies to ALE (such as Eguchi-Hanson), ALF and ALG gravitational instantons. We refer the interested reader to \cite{minerbe2}, \cite{chen-instantons1}, \cite{chen-instantons2}, \cite{chen-instantons3}. 
 
We point out that, to our knowledge, a study of \eqref{laplace-intro} on manifolds with nonnegative Ricci curvature  was yet to be addressed in literature. In the case of subcritical exponent $1 \leq \alpha < (n+2)/(n-2)$ in the right hand side of \eqref{laplace-intro}, it has been shown in \cite[Theorem 1.2]{gidas-spruck} that entire solutions are not allowed. Theorem \ref{main} can thus be interpreted as a first extension of such study to the critical regime.
Concerning the case of complete Riemannian manifolds with other curvature assumptions, we mention \cite{muratori-soave} by Muratori-Soave, who considered solutions to \eqref{laplace-intro} that are radially symmetric or that minimize the Sobolev quotient in Cartan-Hadamard manifolds.

\smallskip

Our proof combines elementary computations on a vector field with nonnegative divergence already considered by \cite{obata-critlap} and \cite{gidas-spruck}, a Harnack-type argument triggered by the decay assumption coupled with a Yau-type gradient estimate, and a novel auxiliary function $V$ defined along the level sets of $u$. In proving Theorem \ref{main}, we first show an intermediate result that adds to the decay assumption on $u$ a \emph{finite energy condition}, that is to say
\begin{equation}
\label{finiteenergy}
\int_M \abs{\nabla u}^2 \dd\mu < + \infty,
\end{equation}
but does not require any volume restriction.
It reads as follows.
\begin{theorem}
\label{main-energy}
Let $(M, g)$ be a complete noncompact Riemannian manifold with nonnegative Ricci curvature and of dimension $n \geq 3$. Assume that there exists a global, strictly positive \emph{finite energy} solution $u$ to 
\[
- \Delta u = n(n-2) u^{\frac{n+2}{n-2}},
\]
such that $u(x) d(o, x)^{(n-2)/2}$ is uniformly bounded in $M$ as a function of $x \in M$.
Then, $(M, g)$ is isometric to $(\R^n, g_{\R^n})$ and 
\begin{equation}
u(x) = \left(\frac{a}{a^2 + d(o, x)^2}\right)^{\frac{(n-2)}{2}},
\end{equation}
for some $a > 0$.
\end{theorem}
It is worth pointing out that, thanks to the decay estimate on the gradient that we observe in  Proposition \ref{decaygrad-prop}, if the decay on $u$ is reinforced to $u \leq C d(o, x)^{(2-n)/2 - \epsilon}$, for $\epsilon > 0$, then \eqref{finiteenergy} is just a consequence of this. In particular, Theorem \ref{main-energy} fully recovers the main result in \cite{gidas-ni-nirenberg}, and extends it to a rigidity statement in the geometry of nonnegative Ricci curvature. 

\medskip

\textbf{Summary.} In Section 2 we gather some basic estimates about $u$ and its derivatives. In Section 3 we introduce the fundamental vector field with nonnegative divergence and set its connection with the auxiliary function $V$, and finally in Section 4 we prove Theorems \ref{main-energy} and  \ref{main}.

\medskip

\begin{center}
{\emph{Added note}}
\end{center}
During the checking process of the presentation of the manuscript, the preprint \cite{catino-monticelli} appeared on the arXiv. It refines and generalizes our results, with different methods.

\section{Preliminary a priori estimates}
In this section we work out some preliminary estimates in force for solutions to \eqref{laplace-intro}, that we later particularize for those such that $u(x) d(o, x)^{(n-2)/2}$ is uniformly bounded. 

First of all, we crucially observe that the technique Yau \cite{yau} pioneered in order to obtain gradient estimates for harmonic functions under lower Ricci curvature bounds yields for solutions to the critical Laplace equation the following inequality.
\begin{proposition}[Yau-type estimate for solutions to the critical Laplace equation]
\label{yau}
Let $(M, g)$ be a complete, noncompact Riemannian manifold with nonnegative Ricci curvature of dimension $n \geq 3$. Then, a solution to \eqref{laplace-intro} satisfies
\begin{equation}
\label{yauf}
\sup_{B(x, R)} \frac{\abs{\nabla u}^2}{u^2} \leq C \left[\frac{1}{R^2} + \sup_{B(x, 2R)} u^{\frac{4}{(n-2)}}\right]
\end{equation}
for any $x \in M$, for any $R \geq 1$ and for some dimensional constant $C > 0$.
\end{proposition}   
The proof of \eqref{yauf} is obtained by following very closely the one for eigenfunctions proposed for \cite[Theorem 1.1]{li-lecture} and, as such, we omit the proof.

As a consequence of the above result, we have the following decay estimate for $\abs{\nabla u}$ when $u$ satisfies the same bound as in Theorem \ref{main}. Here and in the sequel of the paper, we consider $o \in M$ to be any fixed point, and use the notation $r(x) = d(o, x)$ for $x \in M$.
\begin{corollary}[Decay of the gradient]
\label{decaygrad-prop}
Let $(M, g)$ be a complete, noncompact Riemannian manifold with nonnegative Ricci curvature of dimension $n \geq 3$, and let $u$ be a solution to \eqref{laplace-intro} that satisfies $u  \leq C r^{-(n-2)/2}$. Then,
\begin{equation}
\label{decay-grad}
\frac{\abs{\nabla u}^2}{u^2} (x) \leq C \frac{1}{r(x)^2}
\end{equation}
for any $x$ such that $r(x) = d(o, x) \geq 4$.
\end{corollary}
\begin{proof}
It suffices to apply Proposition \ref{yau} with $2R = r(x)/2$. Coupled with the decay assumption on $u$, it yields
\begin{equation}
\label{passaggio-decay-yau}
\frac{\abs{\nabla u}^2}{u^2} \leq C\left[\frac{1}{r(x)^2} + u^{\frac{4}{n-2}}(y_x)\right] \leq C\left[\frac{1}{r(x)^2} + \frac{1}{r(y_x)^2}\right],
\end{equation}
for some $y_x \in B(x, r(x)/2)$. On the other hand, we have $r(x)/3 \leq r(y_x) \leq 2r(x)$, and plugging this information into \eqref{passaggio-decay-yau} leaves us with \eqref{decay-grad}.
\end{proof}
Through integrating by parts the Bochner formula, we do now work out an integral estimate on the Hessian of $u$. The computations, that in our case is particularly simple, got some inspiration from analogous ones in the celebrated work \cite{Che-Cold}. Here and in the remainder of this paper we consider $o \in M$ to be fixed and denote with $A_{R_1, R_2}$ the open annulus $B(o, R_2) \setminus \overline B(o, R_1)$. 
\begin{proposition}
\label{hessian-prop}
Let $(M, g)$ be a complete, noncompact Riemannian manifold with nonnegative Ricci curvature of dimension $n \geq 3$. Then, for any $k > 1$ there exist numbers $1 < \alpha < \beta < k$ such that a solution to \eqref{laplace-intro} satisfies, for any $R \geq 1$,
\begin{equation}
\label{hessianf}
\int\limits_{A_{\alpha R, \beta R}} \abs{\nabla \nabla u}^2 \dd\mu \leq C \left(R^{-2}\int\limits_{A_{R, kR}} \abs{\nabla u}^2 \dd\mu  + \int\limits_{A_{R, kR}}u^{\frac{4}{(n-2)}} \abs{\nabla u}^2 \dd\mu\right),
\end{equation}
for some constant $C$ not depending on $R$.
\end{proposition}
\begin{proof}
By the Bochner formula and by the nonnegativity of the Ricci tensor, we have
\begin{equation}
\label{bochner-decay}
\abs{\nabla \nabla u}^2 \leq \frac{1}{2} \Delta \abs{\nabla u}^2 - \langle\nabla \Delta u, \nabla u\rangle.
\end{equation}
We multiply both sides by an annular Euclidean-like cut-off function
$\phi$ supported in $B(o, kR) \setminus \overline{B(o, R)}$  such that
\begin{align}
\label{annular-cutoff}
\phi \equiv 1 \, \,&\text{on}  \,\,  B(o, \beta R) \setminus \overline{B(o, \alpha R)}, \\
 \abs{\nabla \phi} &\leq C \, \frac{\phi^{1/2}}{R}, \\
 \abs{\Delta \phi} &\leq C \, \frac{1}{R^2}.
\end{align}
for $1 < \alpha < \beta < k$. Such function is well known to exist since \cite{Che-Cold}, and can be more precisely built by slightly re-adapting the proof of \cite[Corollary 2.3]{bianchi-setti}. 
Integrating by parts twice the first term  on the right-hand side, we get
\[
\begin{split}
\int\limits_{A_{\alpha R, \beta R}} \abs{\nabla \nabla u}^2 \dd\mu &\leq \int\limits_{A{r, k R}} \abs{\nabla \nabla u}^2 \phi \dd\mu \\
&\leq \frac{1}{2} \int\limits_{A{r, k R}}\abs{\nabla u^2} \Delta \phi + \int\limits_{A_{R, kR}}u^{\frac{4}{(n-2)}} \abs{\nabla u}^2 \phi \dd\mu \\ 
&\leq C\left(R^{-2}\int\limits_{A_{R, kR}} \abs{\nabla u}^2 \dd\mu  + \int\limits_{A_{R, kR}}u^{\frac{4}{(n-2)}} \abs{\nabla u}^2 \dd\mu\right), 
\end{split}
\]
as claimed.
\end{proof}
When $u$ is assumed to satisfy $u \leq Cr^{-(n-2)/2}$ we get the following integral decay estimate for its Hessian.
\begin{corollary}[Integral decay of the Hessian]
\label{hessian-decay-corollary}
Let $(M, g)$ be a complete, noncompact Riemannian manifold with nonnegative Ricci curvature of dimension $n \geq 3$, and let $u$ be a solution to \eqref{laplace-intro} that satisfies $u  \leq C r^{-(n-2)/2}$. Then, for any $k > 1$ there exist numbers $1 < \alpha < \beta < k$ such that a solution to \eqref{laplace-intro} satisfies, for any $R \geq 1$,
\begin{equation}
\label{decay-hess}
\int\limits_{A_{\alpha R, \beta R}} \abs{\nabla \nabla u}^2 \dd\mu \leq \frac{C}{R^{2}}\int\limits_{A_{R, kR}} \abs{\nabla u}^2 \dd\mu \leq \frac{C}{R^2},
\end{equation}
for some positive constant $C$.
\end{corollary}
\begin{proof}
The first inequality just follows by plugging the assumption on $u$ into \eqref{hessianf}, the second one by coupling it with \eqref{decay-grad}, and observing that, by Bishop-Gromov monotonicity, we have
\[
\abs{A_{r, kr}} \leq \abs{B(o, kR)} \leq \omega_n (kR)^n,
\]
where $\omega_n$ is the Lebesgue measure of the unit ball in $\R^n$.
\end{proof}
We close this preliminary section by recording a basic and fundamental control on $u$ from below.
\begin{proposition}[Control from below for $u$]
Let $(M, g)$ be a complete, noncompact Riemannian manifold with nonnegative Ricci curvature of dimension $n \geq 3$, and let $u$ be a solution to \eqref{laplace-intro} that goes to zero at infinity. Then, there exists a strictly positive constant $C$ such that
\begin{equation}
\label{control-below-u}
u(x) \geq C r(x)^{2-n}
\end{equation}
for any $x \in M \setminus B(o, 1)$.
\end{proposition}
\begin{proof}
Let $C = \min_{\partial B(0, 1)} u$, and recall that $\Delta r^{2-n}$ is sub-harmonic in the sense of distributions in $M \setminus \{o\}$ (see e.g. the proof of \cite[Lemma 2.2]{Ago_Fog_Maz_1}). Then, the function $u/C - r^{2-n}$ is super-harmonic in $B(o, R) \setminus \overline{B(o, 1)}$ and in particular it satisfies
\[
(u/C - r^{2-n})(x) \geq \sup_{\partial B(o, 1)}(u/C - r^{2-n}) + \sup_{\partial B(o, R)}(u/C - r^{2-n}) \geq \sup_{\partial B(o, R)}(u/C - r^{2-n})    
\]
for any $x \in B(o, R) \setminus \overline{B(o, 1)}$, where the last inequality is due to the specific choice of $C$. Since both $u$ and $r^{2- n}$  vanish at infinity, we get that the claim by letting $R$ at infinity and by the arbitrariness of $x$.
\end{proof}

\section{A key vector field and a useful auxiliary function}
In this section we introduce a vector field with nonnegative divergence, vanishing exactly in the case in the flat model situation and set the main properties of the integral auxiliary function $V$. 

\smallskip

We introduce the vector field with nonnegative divergence that rules our problem. Its intuition dates back to the early work of Obata \cite{obata-critlap}, and it constitutes a key ingredient also in the study of subcritical elliptic equations performed in \cite{gidas-spruck}. 
As in the latter, the vector field is better understood in terms of the function 
\begin{equation}
\label{v-def}
v = u^{-\frac{2}{n-2}},
\end{equation}
with $u$ a strictly positive solution to \eqref{laplace-intro}. This is due to the fact that, in the model situation of the flat $\R^n$ with $u$ given by a Talentian \eqref{talentian}, $v$ becomes an affine function of $d(o, x)^2$, that in particular solves
\[
\nabla \nabla v = \frac{\Delta v}{n} g.
\] 
In fact, the squared norm of the trace-free Hessian of $v$ will constitute, together with a Ricci curvature term, the nonnegative divergence of the vector field. 

\smallskip

We define 
\begin{equation}
\label{vector-field}
X = \frac{1}{v^{n-1}}\left[\frac{1}{2}\nabla \abs{\nabla v}^2 -  \frac{1}{2} \frac{\abs{\nabla v}^2\nabla v}{v} - 2 \frac{{\nabla v}}{v}\right].
\end{equation}
We actually give, for the reader's sake, a self-contained computation of its divergence.
In order to do it, we first point out that the critical Laplace equation translates as
\begin{equation}
\label{eq-v}
\Delta v = \frac{n}{2} \frac{\abs{\nabla v}^2}{v} + \frac{2n}{v} 
\end{equation}
in terms of $v$.
\begin{proposition}
\label{field-prop}
Let $(M, g)$ be a Riemannian manifold, $\Omega \subset M$ an open set and $u: \Omega \to \R$ a strictly positive function solving \eqref{laplace-intro}. Then, letting $v$ be as in \eqref{v-def}, the vector field defined by \eqref{vector-field} satisfies
\begin{equation}
\label{divergence}
\dive X = \frac{1}{v^{n-1}}\left[\left\vert \nabla \nabla v - \frac{\Delta v}{n} g \right\vert^2 + \ric \left(\nabla v, \nabla v\right) \right].
\end{equation}
\end{proposition} 
\begin{proof}
We have
\begin{equation}
\label{divergence1}
\begin{split}
\dive X \,=\,  &\frac{1}{v^{n-1}}\left[\frac{1}{2}\dive (\nabla \abs{\nabla v}^2) - \frac{1}{2} \dive \left(\abs{\nabla v}^2 \frac{\nabla v}{v}\right) - 2 \dive\left( \frac{\nabla v}{v}\right)\right] \\
& - \frac{(n-1)}{v^{n-1}}\left[\frac{1}{2}\left\langle\nabla \abs{\nabla v}^2, \frac{\nabla v}{v}\right\rangle - \frac{\abs{\nabla v}^4}{2v^2} - 2\frac{\abs{\nabla v}^2}{v^2} \right].
\end{split}
\end{equation}
We do now compute the three divergence terms. We have, using Bochner identity and \eqref{eq-v},
\begin{equation}
\label{divergence2}
\begin{split}
\dive (\nabla \abs{\nabla v}^2) &= 2\left[\left\vert \nabla \nabla v - \frac{\Delta v}{n}\right\vert^2 + \frac{(\Delta v)^2}{n} + \left\langle \nabla \Delta v, \nabla v\right\rangle + \ric(\nabla v, \nabla v)\right] \\
&= 2\left[\left\vert \nabla \nabla v - \frac{\Delta v}{n}\right\vert^2 - \frac{n}{4}\frac{\abs{\nabla v}^4}{v^2} + \frac{4n}{v^2} + \frac{n}{2} \left\langle\nabla \abs{\nabla v}^2, \frac{\nabla v}{v} \right\rangle + \ric(\nabla v, \nabla v) \right].
\end{split}
\end{equation}
Moreover,
\begin{equation}
\label{divergence3}
\dive \left(\abs{\nabla v}^2 \frac{\nabla v}{v}\right) = \left\langle \nabla \abs{\nabla v}^2, \frac{\nabla v}{v}\right\rangle + \frac{1}{2}(n-2) \frac{\abs{\nabla v}^4}{v^2} + 2n \frac{\abs{\nabla v}^2}{v^2}. 
\end{equation}
Finally,
\begin{equation}
\label{divergence4}
\dive\left( \frac{\nabla v}{v}\right) = \frac{(n - 2)}{2} \frac{\abs{\nabla v}^2}{v^2} + \frac{2n}{v^2}.
\end{equation}
Plugging \eqref{divergence2}, \eqref{divergence3} and \eqref{divergence4} into \eqref{divergence1} leaves us with identity \eqref{divergence}.
\end{proof}
We now define, a priori only for regular level sets of $v$, the function  
\[
V: (\inf_M v, + \infty) \setminus v(\mathrm{Crit}(v)) \to (0+ \infty),
\]
where $\mathrm{Crit}(v) = \{ \nabla v = 0\}$,
given by 
\begin{equation}
\label{V-def}
V(s) = \int\limits_{\{v = s\}}\abs{\nabla v}^3 \dd\sigma. 
\end{equation}
In particular, by Sard's Theorem, this function is defined for almost every $s \in (\inf_M v, + \infty)$.
In the next result, we show that it is actually equivalent to an absolutely continuous function, and compute its derivative. In order to justify the formal computation, we argue as in the recent paper \cite{colding-eigenfunctions}, but we provide, for completeness, all details.
\begin{proposition}
Let $(M, g)$ be a Riemannian manifold, and let $u$ be a positive solution to \eqref{laplace-intro} that tends to zero at infinity. Then, the function $V$ defined above is equivalent to an absolutely continuous function $V : (\inf_M v, + \infty) \to \R$, and moreover
\begin{equation}
\label{V-der}
V'(s) = \int\limits_{\{v = s\}} \left\langle \nabla \abs{\nabla v}^2, \frac{\nabla v}{\abs{\nabla v}} \right\rangle \dd\sigma + \frac{n}{2} s^{-1}\int\limits_{\{v = s\}} \abs{\nabla v}^3 \dd\sigma + 2n s^{-1} \int_{\{v = s\}} \abs{\nabla v} \dd\sigma 
\end{equation}
holds true for almost any value $s \in (\inf_M v, + \infty)$. 
\end{proposition}
\begin{proof}
Let $s$ be a regular value for $v$, so that
\[
V(s) = \int_{\{v = s\}}\abs{\nabla v}^2 \left\langle \nabla v, \frac{\nabla v}{\abs{\nabla v}}\right\rangle \dd\sigma = \int\limits_{\{v < s\}} \dive (\abs{\nabla v}^2 \nabla v) \dd\mu,
\]
where we used the divergence theorem and that, since $v$ tends to infinity, it is a proper function and the outer unit normal to $\{v < s\}$ is given by $\nabla v /\abs{\nabla v}$.
 
Observe now that, arguing exactly as in the basic proof of \cite[Lemma 1.1]{colding-eigenfunctions}, the critical set of $u$ is locally contained in smooth $(n-1)$-dimensional manifolds, in particular it is of null $n$-dimensional Hausdorff measure and moreover all the level sets of $u$ (and hence of $v$) have zero $n$-dimensional Hausdorff measure. Indeed, this argument just relies on the fact that $\Delta u < 0$. 
In particular, the function $\tilde {V} : (\inf_M v, + \infty) \to \R$ defined by 
\[
\tilde{V}(s) = \int_{\{v = s\}}\abs{\nabla v}^2 \left\langle \nabla v, \frac{\nabla v}{\abs{\nabla v}}\right\rangle \dd\sigma = \int\limits_{\{v < s\}} \dive (\abs{\nabla v}^2 \nabla v) \dd\mu,
\]
is immediately seen to be continuous. Indeed, for $\epsilon > 0$, we have
\[
\tilde{V}(s + \epsilon) - \tilde{V}(s) = \int\limits_{\{s \leq v < s + \epsilon\}} \dive (\abs{\nabla v}^2 \nabla v) \dd\mu = \int\limits_{\{s < v < s + \epsilon\}} \dive (\abs{\nabla v}^2 \nabla v) \dd\mu 
\] 
where we have used $\mu\{v = s\} = 0$, and, in particular, the right-hand side above approaches $0$ as $\epsilon \to 0^+$ by the Dominated Convergence Theorem. 

We proved so far that $V$ is equivalent to the continuous function $\tilde{V}$, and in particular we can safely identify one with the other. In order to show the absolute continuity, let, as in the proof of \cite[Lemma 1.3]{colding-eigenfunctions},
\[
V_\delta (s) = \int\limits_{\{v < s\}} \frac{\dive(\abs{\nabla v}^2 {\nabla v})\abs{\nabla v}}{\abs{\nabla v} + \delta} \dd\sigma,
\] 
for $\delta > 0$. Observe that, since as already pointed out $\mu \{\nabla v= 0\} = 0$, the above expression converges by dominated convergence theorem as $\delta \to 0^+$ to $V(s)$. 
Moreover, by coarea formula,
we have
\[
V_\delta (s) = \int\limits_{v_0}^s\int_{\{v = t\}} \frac{\dive(\abs{\nabla v}^2 \nabla v)}{\abs{\nabla v} + \delta} \dd\sigma \dd t,
\]
where we let $v_0 = \inf_M v$.
By the Monotone Convergence Theorem, we can pass to the limit as $m \to\+\infty$ also in the right-hand side of the identity above, and deduce that
\begin{equation}
\label{abs-cont}
V(s) = \int\limits_{v_0}^s\int_{\{v = t\}} \frac{\dive(\abs{\nabla v}^2\nabla v)}{\abs{\nabla v}} \dd\sigma \dd t,
\end{equation}
for any $s \in (\inf_M v, + \infty)$, that is the absolute continuity of $V$. Finally, computing, with the aid of \eqref{eq-v},
\[
\dive(\abs{\nabla v}^2\nabla v) = \left\langle \nabla \abs{\nabla v}^2, {\nabla v}\right\rangle +  \frac{n}{2} \frac{\abs{\nabla v}^4}{v} + 2n \frac{\abs{\nabla v}^2}{v},
\]
and plugging it into \eqref{abs-cont}, we also showed \eqref{V-der}.
\end{proof}
We finally link the function $V$ and its derivative to the vector field $X$. We have, by the Divergence Theorem,
\begin{equation}
\int\limits_{\{v < s\}} \dive X \dd\mu = \frac{\int_{\{v = s\}} \left\langle \nabla \abs{\nabla v}^2, {\nabla v}{\abs{\nabla v}}^{-1} \right\rangle \dd\sigma}{s^{n-1}} -\frac{1}{2}\frac{\int_{\{v = s\}}\abs{\nabla v}^3 \dd\sigma}{s^n} - 2 \frac{\int_{\{v = s\}} \abs{\nabla v} \dd\sigma}{s^{n}}
\end{equation}
for any regular value $s$. Coupling this with \eqref{divergence} and \eqref{v-def}, we deduce the following.

\begin{corollary}
Let $(M, g)$ be a complete Riemannian manifold. Let $u$ be a strictly positive solution to \eqref{laplace-intro} that vanishes at infinity. Then, we have
\begin{equation}
\label{identity-key}
\begin{split}
\int\limits_{\{v < s\}} \frac{1}{v^{n-1}}\left[\left\vert \nabla \nabla v - \frac{\Delta v}{n} g \right\vert^2 + \ric \left(\nabla v, \nabla v\right) \right] \dd\mu = \frac{1}{2}\frac{V'(s)}{s^{n-1}} &- \frac{1}{4}(n+2) \frac{V(s)}{s^n} \\
&- (n+2)\frac{\int_{\{v = s\}} \abs{\nabla v} \dd\sigma}{s^n}
\end{split}
\end{equation}
for almost any $s \in v(M)$.
\end{corollary}
The main aim, for proving Theorems \ref{main-energy} and \ref{main}, will be showing that in their assumptions, the left-hand side of \eqref{identity-key} vanishes for any $s$.
\section{Proof of the main results}
In this section we prove Theorems \ref{main} and \ref{main-energy}. We first prove the latter, i.e. the rigidity result for finite energy solutions, and then prove that in the setting of Theorem \ref{main} the function $u$ actually enjoys finite energy. We start with the following simple splitting principle, consequence of the basic result according to which the existence of a nontrivial function with vanishing trace-less Hessian implies a warped product splitting of the metric \cite[Theorem 5.7.4]{Petersen_book}, \cite{cat-man-maz}, \cite[Section 1]{Che-Cold}.
\begin{lemma}
\label{lemma-splitting}
Let $(M, g)$ be a complete manifold, and let $u$ be a solution to \eqref{laplace-intro} that vanishes at infinity. If 
\begin{equation}
\label{diva0}
\left\vert \nabla \nabla v - \frac{\Delta v}{n} g \right\vert^2 + \ric \left(\nabla v, \nabla v\right) \equiv 0 
\end{equation}
on $M$, then $(M, g)$ is isometric to flat $\R^n$ and $u$ is of the form \eqref{talentian}.
\end{lemma}
\begin{proof}
It is classically deduced from the vanishing of the trace-less Hessian of $v$, e.g. appealing to \cite[Theorem 5.7.4]{Petersen_book}, that $(M, g)$ must split a warped product $(I \times N, \dd\rho \otimes \dd\rho + \phi^2(\rho) g_{N})$, for some hypersurface $N$ in $M$ and with $I$ coinciding with $\R$ or $[0, + \infty)$. By the proof presented in the contribution above (see also the maybe more transparent computations carried out in the proof of \cite[Theorem 1.1]{cat-man-maz}, one also realizes that $N$ can be identified with a smooth level set of $v$, and that in particular it is a closed hypersurface. Again by construction of the splitting, one also has that in these coordinates 
$v'(\rho) = 2\alpha \phi(\rho)$ for some constant $\alpha$. Consequently, since the Ricci curvature of $g$ can be computed as
\begin{equation}
\label{ricci-warped}
\ric
 = 
-(n-1)
\frac{\phi''}{\phi}
\dd \rho \otimes \dd \rho
+ 
\ric_N
-
\big((n-2)(\phi')^2+\phi \phi''\big)^{2}
g_N,
\end{equation}
plugging in $\ric(\nabla v, \nabla v) = \abs{v'(\rho)}^2 \ric (\nabla \rho, \nabla \rho) = 0$
we deduce that $\phi = A + B\rho$,  for some constants $A, B$. If $B = 0$, then $(M, g)$ would be a cylinder with cross-section $N$, and $v = A\rho + C$, for some other constant $C$. Consequently, $\Delta v = 0$, and this is in contradiction with the equation \eqref{eq-v}. 

We can thus suppose, by possibly translating the coordinate $\rho$, that $A =  0$ and $\phi = B \rho $ for some constant $B \neq 0$. In particular, since $g$ must be smooth also as $\rho \to 0^+$, we have $I = [0, +\infty)$, $(N, g_N)$ the sphere $\Sf^{n-1}$ with its standard round metric and $b = 1$.
In other words, $M$ is the Euclidean space and $g$ is its flat norm.

Moreover, $v = \alpha \rho^2 + \beta$, with $\rho$ constituting the Euclidean distance from some origin $o$.  Plugging this function into \eqref{eq-v}, one immediately gets $\alpha\beta = 1$, and thus, by \eqref{V-def}, $u$ has the form \eqref{talentian}.
\end{proof}
 
Exploiting crucially \eqref{identity-key} and the above splitting principle, we prove Theorem \ref{main-energy}. 
  
\begin{proof}[Proof of Theorem \ref{main-energy}]  
Due to the finite energy condition \eqref{finiteenergy} and the coarea formula, we have
\[
\int\limits_0^{t_0} \int_{\{u =t\}} \abs{\nabla u} \dd\sigma \dd t < + \infty. 
\]
for any $t_0 \in u(M)$. Performing a change of variables, such condition translates in terms of $v$ as
\begin{equation}
\label{energy-v}
\int\limits_{s_0}^{+\infty} \int_{\{v=s\}} \frac{\abs{\nabla v}}{s^n} \dd\sigma \dd s < + \infty.
\end{equation}
We want to prove that
\begin{equation}
\label{diva0-inproof}
\int\limits_{\{v < s_j\}} \frac{1}{v^{n-1}}\left[\left\vert \nabla \nabla v - \frac{\Delta v}{n} g \right\vert^2 + \ric \left(\nabla v, \nabla v\right) \right] \dd\mu \to 0^+
\end{equation}
for some $s_j \to +\infty$. Indeed, since the Ricci curvature, and thus the integrand, are nonnegative, this would complete the proof by Lemma \ref{lemma-splitting}.
Assume now  by contradiction that \eqref{diva0-inproof} holds for no diverging sequences $\{s_j\}_{j \in \N}$. Then, by \eqref{identity-key}, we deduce that for almost any $s \in v(M)$
\[
V'(s) \geq \delta s^{n-1}
\]
for some $\delta > 0$. Integrating the above inequality in $(s_0, s)$ for some $s_0 \in v(M)$, we deduce the following growth for the absolutely continuous function $V$
\begin{equation}
\label{growth-v}
V(s) \geq \frac{\delta}{n} s^n + F(s_0),
\end{equation}
with $F(s_0) = V(s_0) - \frac{\delta}{n}{s_0}^n$.
Moreover, by choosing $s_0$ big enough, by coupling \eqref{decay-grad} with \eqref{control-below-u}, one gets
\begin{equation}
\label{decay-gradv1}
\abs{\nabla v}^2 (x) \leq C \frac{v^2}{r^2} (x) \leq C v(x) 
\end{equation}
for any $x \in \{v > s_0\}$. But then, by \eqref{growth-v} and \eqref{decay-gradv1}, the left-hand side of \eqref{energy-v} must also satisfy
\begin{equation}
\label{conto-finale-energy}
\int\limits_{s_0}^{+\infty}\int_{\{v = s\}} \frac{\abs{\nabla v}}{s^n} \dd\sigma \dd s = \int\limits_{s_0}^{+\infty} \int_{\{v = s\}} \frac{\abs{\nabla v}^3}{\abs{\nabla v}^2 s^n} \dd\sigma \dd s
\geq  \frac{1}{C} \int\limits_{s_0}^{+\infty} \frac{V(s)}{s^{n+1}} \dd s 
\geq \frac{1}{C} \int\limits_{s_o}^{+\infty} \frac{(\delta/n) s^n + F(s_0)}{s^{n+1}} \dd s,
\end{equation}
which contradicts \eqref{energy-v}.
\end{proof}
The proof of Theorem \ref{main} fully builds on the following improved decay estimate on $u$. It crucially uses again the vector field $X$, this time integrating its divergence on geodesic balls in place of sub-level sets of $v$, and a Harnack-type argument. 
\begin{proposition}
\label{decay-u-prop}
Let $(M, g)$ and $u$ satisfy the assumptions of Theorem \ref{main}. Then, for any positive $\epsilon < (n-2) /6$, we have
\begin{equation}
\label{decay-epsilon}
u(x) r^{(n-2)/2 + \epsilon} (x) \to 0
\end{equation} 
as $r(x) = d(o,x) \to +\infty$.
\end{proposition}
\begin{proof}
We apply the Divergence Theorem to the vector field $X$ of \eqref{vector-field} in a geodesic ball $B(o, R)$. We get, by \eqref{divergence}, that
\begin{equation}
\label{divergence-balls}
\begin{split}
\int\limits_{B(o, R)} \frac{1}{v^{n-1}}\left[\left\vert \nabla \nabla v - \frac{\Delta v}{n} g \right\vert^2 + \ric \left(\nabla v, \nabla v\right) \right] \dd\mu = & \frac{1}{2} \int\limits_{\partial B(o, R)} \frac{1}{v^{n-1}}\langle\nabla \abs{\nabla v}^2 , \nu\rangle \dd\sigma \\ 
&- \frac{1}{2} \int\limits_{\partial B(o, R)} \frac{1}{v^n}\langle \abs{\nabla v}^2 \nabla v, \nu \rangle \dd\sigma \\ 
&- 2\int\limits_{\partial B(o, R)} \frac{1}{v^n}\langle {\nabla v}, \nu \rangle \dd\sigma ,
\end{split}
\end{equation}
where $\nu$ is the outward unit normal to $\partial B(o, R)$.
We claim that, if \eqref{decay-epsilon} does not hold for some positive $\epsilon < (n-2)/6$ , then there exists a sequence of radii $R_j \to +\infty$, as $j \to +\infty$, such that the right-hand side of \eqref{divergence-balls} for $R = R_j$ vanishes in the limit as $j \to +\infty$. Indeed, in this case, by the nonnegativity of the left hand side, we deduce that its whole integrand is null and deduce by Lemma \ref{lemma-splitting} that  $u$ is (an Euclidean) Talentian function of the form \eqref{talentian}. But this  actually satisfies \eqref{decay-epsilon}, giving a contradiction. 

\smallskip

Assume then that there exists a sequence of points $x_j$, with  $r(x_j) = d(o, x_j) \to + \infty$ as $j \to + \infty$, such that  
\begin{equation}
\label{decay-contradiction}
u(x_j) r^{\frac{(n-2)}{2} + \epsilon} (x_j) \geq C > 0
\end{equation}
for any $x_j$. By \cite[Proposition 0.4]{minerbe-sobolev}, in force by condition \eqref{volume-condition}, there exists $k > 1$ such that for any $R > 1$, for any two points $x, y$ in $\overline{B(o, R)} \setminus B(o, k^{-1}R)$ there exists a geodesic $\gamma_{x, y}$ connecting them \emph{fully contained} in such an annulus. In particular, setting $R_j = r(x_j)$, for any point $y \in \overline{B(o, R_j)} \setminus B(o, k^{-1}R_j)$ we have
\begin{equation}
\label{harnack-step}
\log u(y) - \log u(x_j) \leq \int\limits_0^{\abs{\gamma_{x_j, y}}} \abs{\nabla \log u}(\gamma_{x_j, y}(t)) \dd t \leq C \frac{\abs{\gamma_{x_j, y}}}{R_j}, 
\end{equation} 
where we denoted by $\abs{\gamma_{x_j, y}}$ the length of the geodesic, and where we employed in the last step the decay estimate for the gradient \eqref{decay-grad}. Moreover, as an application of the Bishop-Gromov monotonicity the ratio on the rightmost hand side of \eqref{harnack-step} is uniformly bounded, see the annular diameter estimate \cite[Lemma 1.4]{abresch-gromoll}. Consequently, any point $y \in \overline{B(o, R_j)} \setminus B(o, k^{-1}R_j)$ in fact satisfies \eqref{decay-contradiction}. Coupling this information with \eqref{control-below-u} and translating in terms of $v$, we get
\begin{equation}
\label{control-v-contradiction}
\frac{1}{C} r \leq v \leq C r^{\frac{2\epsilon}{n-2)} + 1}
\end{equation} 
on any annulus $\overline{B(o, R_j)} \setminus B(o, k^{-1}R_j)$, for some constant $C > 0$ not depending on $R_j$. The gradient decay \eqref{decay-grad}, applied to $v$, improves to
\begin{equation}
\label{decay-improved}
\abs{\nabla v}(x) \leq C r(x)^{\frac{2\epsilon}{n-2}}
\end{equation} 
for any $x \in \overline{B(o, R_j)} \setminus B(o, k^{-1}R_j)$, for some $C > 0$ not depending $R_j$. We can thus estimate, for almost any $\hat{R}_j \in (k^{-1}R_j, R_j)$,
\begin{equation}
\label{decay-int-grad}
\int_{\partial B(o, \hat{R}_j)} \frac{\abs{\nabla v}^b}{v^n} \leq \frac{C}{\hat{R}_j^{1- 2b\epsilon/(n-2)}},
\end{equation}
where we used also the lower bound in \eqref{decay-improved} and that $\abs{\partial B(o, R)} \leq \abs{\Sf^{n-1}} R^{n - 1}$ for almost any $R > 0$, again by the Bishop-Gromov monotonicity (of perimeters).
Since $\epsilon < (n-2)/6$, the above integral vanishes as $j \to + \infty$, when $b=3$ and $b = 1$, and this in particular will imply that the second and the third term in \eqref{divergence-balls} vanish along $\hat{R_j}$. We are left to work out an inequality suitable for controlling the first one.
By direct computation, we have
\[
\nabla \nabla v = \frac{n}{2} \frac{\abs{\nabla v}^2}{v} - \frac{2}{(n-2)} v^{\frac{n}{2}} \nabla\nabla u.
\]
Integrating on an annulus $A_{R_1, R_2}$, with generic radii $0 < R_1 < R_2$, we thus get
\begin{equation}
\label{integration-annulus}
\int\limits_{A_{R_1, R_2}} \frac{\abs{\nabla \nabla v}{\abs{\nabla v}}}{v^{n-1}} \dd\mu \leq C \left[\int\limits_{A_{R_1, R_2}} \frac{\abs{\nabla v}^3}{v^n} \dd\mu + \int\limits_{A_{R_1, R_2}}\frac{\abs{\nabla \nabla u}\abs{\nabla v}}{v^{\frac{n}{2} - 1}}\right] \dd\mu. 
\end{equation}
The first term in the right-hand side is estimated, choosing $R_1 = k^{-1}R_j$ and $R_2 = R_j$ and arguing as above with $b = 3$, with $R_j^{6\epsilon/(n-2)}$. We have used in particular that the volume of the annulus is controlled by $R_j^n$, by Bishop-Gromov's inequality. Restrict now to a smaller annulus $A_{\alpha R_j, \beta R_j}$, with $k^{-1} < \alpha < \beta < 1$ such that
\[
\int\limits_{A_{\alpha R_j, \beta R_j}} \abs{\nabla \nabla u}^2 \dd\mu \leq \frac{C}{R_j}^2,
\]
that exists by Corollary \ref{hessian-decay-corollary}. In such an annulus, the second term in the right hand side of \eqref{integration-annulus} can be estimated as
\begin{equation}
\int\limits_{A_{\alpha R_j, \beta R_j}}\frac{\abs{\nabla \nabla u}\abs{\nabla v}}{v^{\frac{n}{2} - 1}} \dd\mu \leq\frac{\left(\int_{A_{\alpha R_j, \beta R_j}}\abs{\nabla \nabla u}^2 \dd\mu \int_{A_{\alpha R_j, \beta R_j}} \abs{\nabla v}^2 \dd\mu\right)^{1/2}}{R_j^{\frac{n}{2}-1}} 
\leq C R_j^{4\epsilon/(n-2)},
\end{equation}
where we have employed the H\"older inequality together with the estimates on $v$ and $\abs{\nabla v}$ recalled above. We can thus deduce that
\begin{equation}
\int\limits_{A_{\alpha R_j, \beta R_j}} \frac{\abs{\nabla \nabla v}{\abs{\nabla v}}}{v^{n-1}} \dd\mu \leq C \left(R_j^{6\epsilon/(n-2)} + R_j^{4\epsilon/(n-2)}\right).
\end{equation}
Through the coarea formula, we can in particular infer the existence, for any $j$, of a set of positive measure $S_j \subset (\alpha R_j, \beta R_j)$ such that 
\begin{equation}
\label{decay-ultimo}
\int\limits_{\partial B(0, \hat{R}_j)} \frac{\abs{\nabla \nabla v}{\abs{\nabla v}}}{v^{n-1}} \dd\sigma \leq C\left(\frac{1}{\hat{R}_j^{1- 6\epsilon/(n-2)}} + \frac{1}{\hat{R}_j^{1 - 4\epsilon/(n-2)}}\right)
\end{equation}
for any $\hat{R}_j \in S_j$. 
Coupling this with \eqref{decay-int-grad} with $b=1$ and $b = 3$, that holds true in particular for almost any $\hat{R}_j \in S_j$, we can thus bound, up to a multiplicative constant, the absolute value of the right hand side of \eqref{divergence-balls} on spheres of such radii $\hat{R}_j$ with
\begin{equation}
\begin{split}
\int\limits_{\partial B(0, \hat{R}_j)} \frac{\abs{\nabla \nabla v}{\abs{\nabla v}}}{v^{n-1}} \dd\sigma + \int\limits_{\partial B(o, \hat{R}_j)} \frac{\abs{\nabla v}^3}{v^n} \dd\sigma + \int\limits_{\partial B(o, \hat{R}_j)} \frac{\abs{\nabla v}}{v^n} \dd\sigma \leq C\Bigg(\frac{1}{\hat{R}_j^{1 - 6\epsilon/(n-2)}} &+  \frac{1}{\hat{R}_j^{1 - 4\epsilon/(n-2)}} \\ 
&+ \frac{1}{\hat{R}_j^{1 - 2\epsilon/(n-2)}}\Bigg),
\end{split}
\end{equation}
that vanishes as $j \to +\infty$ since $\epsilon < (n-2) /6$, completing the proof.
\end{proof}
Theorem \ref{main} now follows as an immediate corollary of Theorem \ref{main-energy} and Proposition \ref{decay-u-prop}.
\begin{proof}[Proof of Theorem \ref{main}]
By combining \eqref{decay-epsilon} for some $0 < \epsilon < (n-2)/6$ with \eqref{decay-grad}, we deduce that
\[
\abs{\nabla u}^2 \leq C r(x)^{- n - 2\epsilon}
\]
for any $x \in M\setminus B(o, R)$, with $R$ big enough. Consequently, we have
\[
\int\limits_{M \setminus B(o, R)} \abs{\nabla u}^2 \dd\mu \leq C \int\limits_R^{+\infty} \int_{B(o, r)} r^{- n - 2\epsilon} \dd\sigma \dd r \leq C \int\limits_R^{+\infty} \frac{1}{r^{1+2\epsilon}} < + \infty, 
\]
where we have used the coarea formula and $\abs{B(o, r)} \leq \abs{\Sf^{n-1}} r^{n-1}$ for almost any $r > 0$. We conclude by Theorem \ref{main-energy}. 
\end{proof} 


\printbibliography 

@article {abresch-gromoll,
    AUTHOR = {Abresch, U. and Gromoll, D.},
     TITLE = {On complete manifolds with nonnegative {R}icci curvature},
   JOURNAL = {J. Amer. Math. Soc.},
  FJOURNAL = {Journal of the American Mathematical Society},
    VOLUME = {3},
      YEAR = {1990},
    NUMBER = {2},
     PAGES = {355--374},
      ISSN = {0894-0347},
   MRCLASS = {53C21},
  MRNUMBER = {1030656},
MRREVIEWER = {Ji-Ping Sha},
       DOI = {10.2307/1990957},
       URL = {https://doi.org/10.2307/1990957},
}

@incollection {gidas-ni-nirenberg,
    AUTHOR = {Gidas, B. and Ni, W. M. and Nirenberg, L.},
     TITLE = {Symmetry of positive solutions of nonlinear elliptic equations
              in {${\bf R}^{n}$}},
 BOOKTITLE = {Mathematical analysis and applications, {P}art {A}},
    SERIES = {Adv. in Math. Suppl. Stud.},
    VOLUME = {7},
     PAGES = {369--402},
 PUBLISHER = {Academic Press, New York-London},
      YEAR = {1981},
   MRCLASS = {35J60 (53C05 58G20)},
  MRNUMBER = {634248},
MRREVIEWER = {D. E. Edmunds},
}

@article {Ago_Fog_Maz_1,
    author={Agostiniani, V.
and Fogagnolo, M.
and Mazzieri, L.},
title={Sharp geometric inequalities for closed hypersurfaces in manifolds with nonnegative Ricci curvature},
journal={Inventiones mathematicae},
year={2020},
month={Jul},
day={23},
issn={1432-1297},
doi={10.1007/s00222-020-00985-4},
url={https://doi.org/10.1007/s00222-020-00985-4}
}

@article {bianchi-setti,
    AUTHOR = {Bianchi, D. and Setti, A. G.},
     TITLE = {Laplacian cut-offs, porous and fast diffusion on manifolds and
              other applications},
   JOURNAL = {Calc. Var. Partial Differential Equations},
  FJOURNAL = {Calculus of Variations and Partial Differential Equations},
    VOLUME = {57},
      YEAR = {2018},
    NUMBER = {1},
     PAGES = {Paper No. 4, 33},
      ISSN = {0944-2669},
   MRCLASS = {53C21 (35K55 35R01 58J35 58J65)},
  MRNUMBER = {3735744},
MRREVIEWER = {Yu-Zhao Wang},
       DOI = {10.1007/s00526-017-1267-9},
       URL = {https://doi.org/10.1007/s00526-017-1267-9},
}

@article {caffarelli-gidas-spruck,
    AUTHOR = {Caffarelli, L. A. and Gidas, B. and Spruck, J.},
     TITLE = {Asymptotic symmetry and local behavior of semilinear elliptic
              equations with critical {S}obolev growth},
   JOURNAL = {Comm. Pure Appl. Math.},
  FJOURNAL = {Communications on Pure and Applied Mathematics},
    VOLUME = {42},
      YEAR = {1989},
    NUMBER = {3},
     PAGES = {271--297},
      ISSN = {0010-3640},
   MRCLASS = {35J60 (35B40 58G30)},
  MRNUMBER = {982351},
MRREVIEWER = {Robert McOwen},
       DOI = {10.1002/cpa.3160420304},
       URL = {https://doi.org/10.1002/cpa.3160420304},
}

@article {cat-man-maz,
    AUTHOR = {Catino, G. and Mantegazza, C. and Mazzieri, L.},
     TITLE = {On the global structure of conformal gradient solitons with
              nonnegative {R}icci tensor},
   JOURNAL = {Commun. Contemp. Math.},
  FJOURNAL = {Communications in Contemporary Mathematics},
    VOLUME = {14},
      YEAR = {2012},
    NUMBER = {6},
     PAGES = {1250045, 12},
      ISSN = {0219-1997},
   MRCLASS = {53C44 (53A30 53C24 53C25)},
  MRNUMBER = {2989649},
MRREVIEWER = {Otis Chodosh},
       DOI = {10.1142/S0219199712500459},
       URL = {https://doi.org/10.1142/S0219199712500459},
}

@article {Che-Cold,
    AUTHOR = {Cheeger, J. and Colding, T. H.},
     TITLE = {Lower bounds on {R}icci curvature and the almost rigidity of
              warped products},
   JOURNAL = {Ann. of Math. (2)},
  FJOURNAL = {Annals of Mathematics. Second Series},
    VOLUME = {144},
      YEAR = {1996},
    NUMBER = {1},
     PAGES = {189--237},
      ISSN = {0003-486X},
   MRCLASS = {53C21 (53C20 53C23)},
  MRNUMBER = {1405949},
MRREVIEWER = {Joseph E. Borzellino},
       DOI = {10.2307/2118589},
       URL = {https://doi.org/10.2307/2118589},
}

@misc{colding-eigenfunctions,
      title={Optimal growth bounds for eigenfunctions}, 
      author={T. H. Colding and W. P. Minicozzi},
      year={2021},
      eprint={2109.04998},
      archivePrefix={arXiv},
      primaryClass={math.DG}
}

@misc{chen-instantons1,
    title={Gravitational instantons with faster than quadratic curvature decay (I)},
    author={Chen, G. and Chen, X.},
    year={2015},
    eprint={1505.01790},
    archivePrefix={arXiv},
    primaryClass={math.DG}
}

@article{chen-instantons2,
author = {Chen, G. and Chen, X.},
year = {2015},
month = {05},
pages = {},
title = {Gravitational instantons with faster than quadratic curvature decay (II)},
journal = {Journal für die reine und angewandte Mathematik (Crelles Journal)},
doi = {10.1515/crelle-2017-0026}
}

@article{chen-instantons3,
author = {Chen, G. and Chen, X.},
year = {2016},
month = {03},
pages = {},
title = {Gravitational instantons with faster than quadratic curvature decay (III)},
journal = {Mathematische Annalen},
doi = {10.1007/s00208-020-01984-9}
}

@article{muratori-soave,
   title={Some rigidity results for Sobolev inequalities and related PDEs on Cartan-Hadamard manifolds},
   ISSN={0391-173X},
   url={http://dx.doi.org/10.2422/2036-2145.202105_071},
   DOI={10.2422/2036-2145.202105_071},
   journal={Annali Scuola Normale Superiore - Classe di Scienze},
   publisher={Scuola Normale Superiore - Edizioni della Normale},
   author={Muratori, M. and Soave, N.},
   year={2021},
   month={Dec},
   pages={30} }

@article {aubin-isoperimetric,
    AUTHOR = {Aubin, T.},
     TITLE = {Probl\`emes isop\'{e}rim\'{e}triques et espaces de {S}obolev},
   JOURNAL = {J. Differential Geometry},
  FJOURNAL = {Journal of Differential Geometry},
    VOLUME = {11},
      YEAR = {1976},
    NUMBER = {4},
     PAGES = {573--598},
      ISSN = {0022-040X},
   MRCLASS = {58D15 (46E35 53C20)},
  MRNUMBER = {448404},
MRREVIEWER = {J. L. Kazdan},
       URL = {http://projecteuclid.org/euclid.jdg/1214433725},
}

@misc{catino-monticelli,
    title={Semilinear elliptic equations on manifolds with nonnegative Ricci curvature},
    author={Catino, G. and Monticelli, D. D.},
    year={2022},
    eprint={2203.03345},
    archivePrefix={arXiv},
    primaryClass={math.DG}
}

@article {gidas-spruck,
    AUTHOR = {Gidas, B. and Spruck, J.},
     TITLE = {Global and local behavior of positive solutions of nonlinear
              elliptic equations},
   JOURNAL = {Comm. Pure Appl. Math.},
  FJOURNAL = {Communications on Pure and Applied Mathematics},
    VOLUME = {34},
      YEAR = {1981},
    NUMBER = {4},
     PAGES = {525--598},
      ISSN = {0010-3640},
   MRCLASS = {35J60 (53C20 58G05 81E10)},
  MRNUMBER = {615628},
       DOI = {10.1002/cpa.3160340406},
       URL = {https://doi.org/10.1002/cpa.3160340406},
}

@article {li-lecture,
    AUTHOR = {Li, P.},
     TITLE = {Lectures on Harmonic Function},
   NOTE = {Lectures at UCI}
}

@article {minerbe2,
    AUTHOR = {Minerbe, V.},
     TITLE = {On the asymptotic geometry of gravitational instantons},
   JOURNAL = {Ann. Sci. \'Ec. Norm. Sup\'er. (4)},
  FJOURNAL = {Annales Scientifiques de l'\'Ecole Normale Sup\'erieure. Quatri\`eme
              S\'erie},
    VOLUME = {43},
      YEAR = {2010},
    NUMBER = {6},
     PAGES = {883--924},
      ISSN = {0012-9593},
   MRCLASS = {53C26 (53C80)},
  MRNUMBER = {2778451},
MRREVIEWER = {Derek G. Harland},
       DOI = {10.24033/asens.2135},
       URL = {https://doi.org/10.24033/asens.2135},
}

@article {minerbe-sobolev,
    AUTHOR = {Minerbe, V.},
     TITLE = {Weighted {S}obolev inequalities and {R}icci flat manifolds},
   JOURNAL = {Geom. Funct. Anal.},
  FJOURNAL = {Geometric and Functional Analysis},
    VOLUME = {18},
      YEAR = {2009},
    NUMBER = {5},
     PAGES = {1696--1749},
      ISSN = {1016-443X},
   MRCLASS = {53C21 (46E35 53C25)},
  MRNUMBER = {2481740},
MRREVIEWER = {Nikos Labropoulos},
       DOI = {10.1007/s00039-009-0701-3},
       URL = {https://doi.org/10.1007/s00039-009-0701-3},
}

@article {obata-critlap,
    AUTHOR = {Obata, M.},
     TITLE = {The conjectures on conformal transformations of {R}iemannian
              manifolds},
   JOURNAL = {J. Differential Geometry},
  FJOURNAL = {Journal of Differential Geometry},
    VOLUME = {6},
      YEAR = {1971/72},
     PAGES = {247--258},
      ISSN = {0022-040X},
   MRCLASS = {53C20},
  MRNUMBER = {303464},
MRREVIEWER = {Y. Tashiro},
       URL = {http://projecteuclid.org/euclid.jdg/1214430407},
}

@book {Petersen_book,
AUTHOR = {Petersen, P.},
     TITLE = {Riemannian geometry},
    SERIES = {Graduate Texts in Mathematics},
    VOLUME = {171},
   EDITION = {Third},
 PUBLISHER = {Springer, Cham},
      YEAR = {2016},
     PAGES = {xviii+499},
      ISBN = {978-3-319-26652-7; 978-3-319-26654-1},
   MRCLASS = {53-01 (53C20 53C21 53C23)},
  MRNUMBER = {3469435},
       DOI = {10.1007/978-3-319-26654-1},
       URL = {https://doi.org/10.1007/978-3-319-26654-1},
}

@article {talenti,
    AUTHOR = {Talenti, G.},
     TITLE = {Best constant in {S}obolev inequality},
   JOURNAL = {Ann. Mat. Pura Appl. (4)},
  FJOURNAL = {Annali di Matematica Pura ed Applicata. Serie Quarta},
    VOLUME = {110},
      YEAR = {1976},
     PAGES = {353--372},
      ISSN = {0003-4622},
   MRCLASS = {46E35},
  MRNUMBER = {0463908},
MRREVIEWER = {L. Cattabriga},
       DOI = {10.1007/BF02418013},
       URL = {https://doi.org/10.1007/BF02418013},
}

@article {yau,
    AUTHOR = {Yau, S.- T.},
     TITLE = {Harmonic functions on complete {R}iemannian manifolds},
   JOURNAL = {Comm. Pure Appl. Math.},
  FJOURNAL = {Communications on Pure and Applied Mathematics},
    VOLUME = {28},
      YEAR = {1975},
     PAGES = {201--228},
      ISSN = {0010-3640},
   MRCLASS = {53C20 (31C05)},
  MRNUMBER = {0431040},
MRREVIEWER = {Yoshiaki Maeda},
       DOI = {10.1002/cpa.3160280203},
       URL = {https://doi.org/10.1002/cpa.3160280203},
}
\end{document}